\documentclass{amsart}

\usepackage{amsfonts}
\usepackage{amssymb}
\usepackage{amsthm}
\usepackage{amsmath}
\usepackage{enumerate}
\usepackage{hyperref}
\usepackage{amsrefs}
\usepackage{graphicx}

\newcommand{\IR}{\mathbb{R}}
\newcommand{\IC}{\mathbb{C}}
\newcommand{\D}{\mathcal{D}}
\newcommand{\Inn}{\mathcal{I}nn}
\newcommand{\norm}[1]{\left\Vert#1\right\Vert}

\newtheorem{proposition}{Proposition}[section]
\newtheorem{lemma}[proposition]{Lemma}
\newtheorem{theorem}{Theorem}
\newtheorem{corollary}[proposition]{Corollary}
\theoremstyle{definition}

\theoremstyle{remark}

\newtheorem{remark}[proposition]{Remark}

\numberwithin{equation}{section}

\newcommand{\TR}{\hbox{TR}\, }
\newcommand{\tr}{\hbox{tr}\, }
  \newcommand{\tpc}[3]{\left\{#1,#2,#3\right\}}

\newtheorem{problem}{Problem}

 \newcommand{\pair}[2]{\langle{#1,#2}\rangle}

\def\J#1#2#3{ \left\{ #1,#2,#3 \right\} }
\begin{document}

\title{Triple derivations on von Neumann algebras}
\author{Robert Pluta}
\address{Department of Mathematics, University of Iowa, Iowa City, IA 52242-1419, USA}
\email{robert-pluta@uiowa.edu}
\author{Bernard Russo}
\address{Department of Mathematics, University of California, Irvine, CA 92697-3875, USA}
\email{brusso@math.uci.edu}

\date{\today}

\begin{abstract}
It is well known that every derivation of a von Neumann algebra into itself is an inner derivation 
and that every derivation of a von Neumann algebra into its predual is inner.
It is less well known that every triple derivation (defined below) of a von Neumann algebra into itself is an inner triple derivation. 

We examine to what extent all triple derivations of a von Neumann algebra into its predual are inner.  This rarely happens but it comes close.
We prove a (triple)  cohomological characterization of finite factors and a zero-one law for factors.

 \end{abstract}
\maketitle
\tableofcontents
\section{Introduction}

Building on earlier work of Kadison \cite{Kadison66}, Sakai \cite{Sakai66} proved that
very derivation of a von Neumann algebra into itself is 
 inner.   
 Building on earlier work of Bunce and Paschke \cite{BunPas80}, Haagerup, on his way to proving that every C$^*$-algebra is weakly amenable,  showed in \cite{Haagerup83}
 that every derivation of a von Neumann algebra into its predual is inner.
 Thus the first  Hochschild cohomology groups $H^1(M,M)$ and $H^1(M,M_*)$ vanish for any von Neumann algebra $M$.

  It is also known that every triple derivation of a von Neumann algebra into itself is 
 an inner triple derivation (\cite[Theorem 2]{HoMarPerRus02}), and every Jordan derivation of a von Neumann algebra into itself is an inner Jordan derivation \cite[Theorem 3.10]{Upmeier80}.  
 In \cite{PerRus10}, a study was made of automatic continuity of triple derivations from a Banach Jordan triple system into a Banach Jordan triple module, and in \cite{HoPerRus12},    the study of ternary weak amenability in operator algebras and triples was initiated. Triple derivations and inner triple derivations into a triple module are defined and discussed in \cite{HoPerRus12} and in \cite{PerRus10}.\footnote{In this paper, we use the terms `triple' and `ternary' interchangeably, while mindful that in some quarters `triple' means  `Jordan triple' and `ternary' refers to an associative triple setting, such as a TRO (ternary ring of operators)}  However, this paper is essentially independent of  \cite{HoPerRus12} and \cite{PerRus10} as we shall only be concerned with these concepts for von Neumann algebras, with their associated Jordan structure.
 
 Among other things, it was shown in \cite{HoPerRus12} that
every commutative (real or complex) C$^*$-algebra $A$ is 
{\bf ternary weakly amenable},  that is, every triple derivation of $A$ into its dual $A^*$ is an inner triple derivation,  but the C$^*$-algebras $K(H)$ of all compact operators and $B(H)$ of all bounded operators on
an infinite dimensional Hilbert space $H$  do not have this property.


 Two consequences of \cite{HoPerRus12} are that  finite dimensional von Neumann algebras and abelian von Neumann algebras  have the  property that every triple derivation into the predual is an inner triple derivation, analogous to the Haagerup result.  We call this property {\bf normal ternary weak amenability}, and we show that it rarely holds in a general von Neumann algebra, but that it comes close. 
 
 Our main results are Theorems 1 and 2. If the set of inner triple derivations 
from a  von Neumann algebra $M$ into its predual is norm dense in the real vector space 
 of all triple derivations, then $M$ must be finite (Theorem~\ref{thm:0411131}(a)). The converse holds if $M$ acts on a separable Hilbert space, or is a factor (Theorem~\ref{thm:0411131}(b)).  For all  infinite factors,
the norm closure of the set of all inner triple derivations into the predual    has codimension 1 in the space of all triple derivations (Theorem~\ref{thm:0416131}), 
thus providing a  cohomological characterization of finite factors and a zero-one law for factors.

 Inner triple 
derivations  on a von Neumann algebra $M$  into its predual $M_*$ are closely related to the span of commutators of normal functionals with elements of $M$, denoted by $[M_*,M]$. 
 As noted above, for a finite factor of  type $I_n$, all triple derivations into the predual are inner triple derivations.
\ A consequence of Proposition~\ref{prop:0418131} and the  work of Dykema and Kalton \cite{DykKal05} is that no factor of type $II_1$ is normally ternary weakly amenable (Corollary~\ref{cor:0821131}), even though each triple derivation into the predual is a norm limit of inner triple derivations by Theorem~\ref{thm:0411131}.  

We also show that a finite countably decomposable von Neumann algebra $M$
of type $I_n$ is normally ternary weakly amenable   if and only if every element of $M_*$ of central trace zero belongs to $[M_*,M]$. In particular, for such an algebra $M$ with a faithful normal finite trace,  if every element of $M_*$ of  trace zero belongs to $[M_*,M]$ then $M$ is normally ternary weakly amenable.  



\section{Triple modules and triple derivations}


\smallskip
\subsection{Varieties of modules}
Let $A$ be an associative algebra. Let us recall that an
{\bf $A$-bimodule} is a vector space $X$, equipped with two
bilinear products $(a,x)\mapsto a x$ and $(a,x)\mapsto x a$ from
$A\times X$ to $X$ satisfying the following axioms: $$a (b x) = (a
b) x ,\ \ a (x b) = (a x) b, \hbox{ and, }(xa) b = x (a b),$$ for
every $a,b\in A$ and $x\in X$.

Let $A$ be a Jordan algebra, that is, a commutative algebra with product denoted by $\circ$ satisfying $a^2\circ (a\circ b)=a\circ (a^2\circ b)$. A {\bf Jordan $A$-module} is a
vector space $X$, equipped with two bilinear products
$(a,x)\mapsto a \circ x$ and $(a,x)\mapsto x \circ a$ from
$A\times X$ to $X$, satisfying: $$a \circ x = x\circ a,\ \ a^2
\circ (x \circ a) = (a^2\circ  x)\circ a, \hbox{ and, }$$ $$2(
(x\circ a)\circ  b) \circ a + x\circ (a^2 \circ b) = 2 (x\circ
a)\circ  (a\circ b) + (x\circ b)\circ a^2,$$ for every $a,b\in A$
and $x\in X$. \smallskip

In particular, any von Neumann algebra $M$ is a Jordan algebra under the product $a\circ b=(ab+ba)/2$\  and\ 
its  predual $M_*$ is an $M$-bimodule under the actions 
\[
a\varphi(x)=\varphi(xa)\ ,\  \varphi a(x)=\varphi(ax),\hbox{ for }\varphi\in M_*,a,x\in M.
\]
$M_*$ is also a  Jordan $M$-bimodule under the actions 
\[
a\circ \varphi(x)=\varphi(x\circ a)  =\varphi \circ a(x),\hbox{ for }\varphi\in M_*,a,x\in M.
\]

\smallskip

A complex (resp., real) {\bf Jordan triple} is a complex (resp., real)
vector space $E$ equipped with a  triple
product $$ E \times E \times E \rightarrow E\quad\quad
(xyz) \mapsto \J xyz $$
which is bilinear and symmetric in the outer variables and
conjugate linear (resp., linear) in the middle one and satisfying the
so-called \emph{``Jordan Identity''}:
$$
D(a,b) D(x,y) -  D(x,y) D(a,b) = D(D(a,b)x,y) - D(x,D(b,a)y),
$$
for all $a,b,x,y$ in $E$, where $D(x,y) z := \J xyz$. 

A von Neumann algebra $M$ is a complex Jordan triple under the triple product
$\{abc\}=(ab^*c+cb^*a)/2$ and its  predual $M_*$ is a  Jordan triple $M$-module, according to the general  definition\footnote{We do not need the general definition since we are only interested in the special case considered here of a von Neumann algebra and its Jordan structure. It is worth noting however that the triple module actions are linear in some variables and conjugate-linear in other variables.   For example, (\ref{eq module product dual
2}) is conjugate linear in $a,b$ and $\varphi$.} in \cite{PerRus10} and \cite{HoPerRus12}, if we define module actions as follows:

\begin{equation}\label{eq module product dual 1}   \J
ab{\varphi} (x) = \J {\varphi}ba (x) := \varphi \J bax
\end{equation} and
 \begin{equation}\label{eq module product dual
2} \J a{\varphi}b (x) := \overline{ \varphi \J axb },  \end{equation}
for every $ x,a,b\in M,\varphi\in M_*.$
In this case, it is a tedious exercise to show that the identity   
$$\J {a}{b}{\J cde} = \J{\J abc}de 
- \J c{\J bad}e +\J cd{\J abe},$$
which is equivalent to the Jordan identity above,
holds whenever exactly one of the elements belongs to $M_*$.

\medskip
\subsection{Varieties of derivations}
Let $X$ be a Banach $A$-bimodule over an (associative) Banach algebra $A$. A linear mapping $D : A \to X$ is said to be
a {\bf derivation} if $D(a b) = D(a) b + a D(b)$, for every $a,b$ in $A$. For emphasis we call this a  {\bf binary (or associative) derivation}. 
We denote the set of all continuous binary derivations from $A$ to $X$ by $\mathcal{D}_b(A,X)$ .

When $X$ is a Jordan Banach module over a Jordan Banach algebra $A$, a linear mapping $D : A \to X$ is said to be
a {\bf derivation} if $D(a \circ b) = D(a) \circ b + a \circ D(b)$, for every $a,b$ in $A$. For emphasis we call this a  {\bf Jordan derivation}. 
We denote the set of continuous Jordan derivations from $A$ to $X$ by  $\mathcal{D}_J(A,X)$.

In the setting of Jordan Banach triples, a {\bf triple} or {\bf ternary derivation} from a (real or complex)
Jordan Banach triple, $E,$ into a Banach triple $E$-module, $X$,
is a {\it conjugate} linear mapping $\delta: E \to X$ satisfying \begin{equation}\label{eq triple derivation} \delta \J abc =
\J {\delta (a)}bc +  \J a{\delta (b)}c + \J ab{\delta (c)},\end{equation} for every $a,b,c$ in $E$. 
We denote the set of all continuous
ternary derivations from $E$ to $X$ by $\mathcal{D}_t(E,X)$. 

The conjugate linearity (as opposed to linearity) of ternary derivations from a complex Jordan triple into a Jordan triple module is a reflection of the fact that a complex Jordan triple is not necessarily a Jordan triple module over itself (as explained in \cite{PerRus10} and \cite{HoPerRus12}).  This anomaly  has no effect on our results.

Let $X$ be a Banach $A$-bimodule over an associative Banach algebra $A$.
Given $x_{_0}$ in $X$, the mapping $D_{x_{_0}} : A \to X$,
$D_{x_{_0}} (a) = x_{_0} a - a x_{_0}$ is a bounded
(associative or binary) derivation. Derivations of this form are called
{\bf inner}.  We shall use the customary notation $\hbox{ad}\, x_0$ for these inner derivations.
\
The set of all inner derivations from $A$ to $X$
will be denoted by $\mathcal{I}nn_{b} (A,X).$

When $x_{_0}$ is an element in a Jordan Banach $A$-module, $X,$
over a Jordan Banach algebra, $A$, for each $b\in A$, the mapping
$$\delta_{x_{_0},b}=L(x_0)L(b)-L(b)L(x_0) : A \to X,$$ $$\delta_{x_{_0},b} (a) := (x_{_0}
\circ a)\circ b - (b \circ a)\circ x_{_0}, \ (a\in A),$$ is a
bounded derivation.  Here $L(x_0)$ (resp.\ $L(b)$) denotes the module action $a\mapsto x_0\circ a$ (resp.\   multiplication $a\mapsto b\circ a$).  Finite sums of derivations of this form are
called {\bf inner}. 
The set of all inner Jordan derivations from $A$ to
$X$ is denoted by
 $\mathcal{I}nn_{J} (A,X)$.

Let $E$ be a complex (resp., real)
Jordan triple and let $X$ be a triple $E$-module. For each $b\in E$ and each $x_{_0}\in X,$
it is known that
the mapping $$\delta=\delta(b,x_{_0})=L(b,x_0)-L(x_0,b): E  \to X,$$ defined by
\begin{equation}\label{eq:0308111}
\delta (a)=\delta(b,x_{_0}) (a):=\J{b}{x_{_0}}{a}-\J{x_{_0}}{b}{a} \ \ (a\in E),
\end{equation}
is a ternary derivation from $E$ into $X$.
Finite sums of derivations of the form $\delta(b,x_{_0})$ are
called {\bf inner triple derivations}.
The set of all inner ternary
derivations from $E$ to $X$
is denoted by
$\mathcal{I}nn_{t} (E,X)$. 
We shall only need this definition in the case that $E$ is a von Neumann algebra $M$ and $X$ is $M_*$, with module action defined by (\ref{eq module product dual 1}) and (\ref{eq module product dual 2}).\smallskip

The reader who is  interested in more details concerning the more general definitions of the concepts in this section can consult \cite{PerRus10,HoPerRus12}.

\section{Normal ternary weak amenability for factors}

We shall use the facts that if $M$ is a finite von Neumann algebra, then every element of $M$ of central trace zero is a finite sum of commutators (\cite[Theoreme 3.2]{FacdelaHar80},\cite[Theorem 1]{PeaTop69}), and if $M$ is  properly infinite (no finite central projections), then every element of $M$  is a finite sum of commutators (\cite{HalpernTAMS69},\cite[Theorem 1]{HalmosAJM52},\cite[Corollary to Theorem 8]{HalmosAJM54},\cite[Lemma 3.1]{BroPeaTopDuke68}).   Thus for any von Neumann algebra, we have $M=Z(M)+[M,M]$, where $Z(M)$ is the center of $M$ and $[M,M]$ is the set of finite sums of commutators in $M$.

\subsection{Cohomological characterization of finiteness}

Let $M$ be a  von Neumann algebra and consider the submodule $M_*\subset M^*$.  Then
\begin{equation}\label{eq:0822131}
{\mathcal D}_t(M,M_*)={\mathcal I}nn_b^*(M,M_*)\circ *+{\mathcal I}nn_t(M,M_*),
\end{equation}
where for a linear operator $D$, $D\circ *\, (x)=D(x^*)$.
This was  stated and proved for $M$ semifinite in \cite[Cor.\ 3.10]{HoPerRus12} but the same proof holds by using \cite[Theorem 4.1]{Haagerup83} in place of \cite[Theorem 3.2]{BunPas80}. 



Let $M$ be any von Neumann algebra and let 
$\phi_0$ be any fixed normal state. Then \begin{equation}\label{eq:0822132}
M_*=\{1\}_\perp+\IC\phi_0,
\end{equation} where
$$\{1\}_\perp=\{\psi\in M_*:\psi(1)=0\}.$$

\begin{lemma}\label{lem:0411131}
 If $M$ is a von Neumann algebra, then $\overline{[M_*,M]}=Z(M)_\perp$ (norm closure), where $Z(M)$ is the center of $M$. In particular, if $M$ is a factor, then $\overline{[M_*,M]}=\{1\}_\perp$.
\end{lemma}
\begin{proof}  It is clear that $[M_*,M]\subset Z(M)_\perp$.
If $x\in M$ satisfies $\pair{x}{[\varphi,b]}=0$ for all $\varphi\in M_*$ and $b\in M$, then $\varphi(bx-xb)=0$, and so $x$ belongs to the center of $M$ and $\pair{x}{Z(M)_\perp}=0$, proving the lemma.
\end{proof}

\begin{lemma}\label{lem:0411132}
 If $M$ is a properly infinite von Neumann algebra, if $\psi\in M_*$ 
  and if $D_\psi\circ *$ belongs to the norm closure of $\Inn_t(M,M_*)$, where $D_\psi=\hbox{ad}\, \psi$, then $\psi(Z(M))=0$, and in particular, $\psi(1)=0$.
\end{lemma}
\begin{proof}  
For $\epsilon>0$, there exist $\varphi_j\in M_*$ and $b_j\in M$ such that
$$\|D_\psi\circ *-\sum_{j=1}^n (L(\varphi_j,b_j)-L(b_j,\varphi_j))\|<\epsilon.$$

For $x,a\in M$, direct calculations yield
$$\left| \psi(a^*x-xa^*)-\frac{1}{2}\sum_{j=1}^n(\varphi_jb_j-b_j^*\varphi_j^*)(a^*x)
-\frac{1}{2}\sum_{j=1}^n(b_j\varphi_j-\varphi_j^*b_j^*)(xa^*)\right |<\epsilon\|a\|\|x\|.$$

We  set $x=1$  to get
\[
\left |\frac{1}{2}\sum_{j=1}^n(\varphi_jb_j-b_j^*\varphi_j^*)(a^*)
+\frac{1}{2}\sum_{j=1}^n(b_j\varphi_j-\varphi_j^*b_j^*)(a^*)\right |<\epsilon \|a\|,
\]
and therefore
$$
\left | \psi(a^*x-xa^*)-\frac{1}{2}
\sum_{j=1}^n(\varphi_jb_j-b_j^*\varphi_j^*)(a^*x-xa^*)\right | <2\epsilon \|a\|\|x\|,$$ for every $a,x\in M$, that is,
\begin{equation}\label{eq:511}\left |\psi([a,x])-\frac{1}{2}\sum_{j=1}^n(\varphi_jb_j-b_j^*\varphi_j^*)([a,x])\right |<2\epsilon\|a\|\|x\|,\end{equation}
and therefore
\begin{equation}\label{eq:512}
\left |\psi([a,x])-\frac{1}{2}\sum_{j=1}^n(\varphi_j^*b_j^*-b_j\varphi_j)([a,x])\right |<3\epsilon\|a\|\|x\|.
\end{equation}
Let us now write
$$\sum_j(\varphi_jb_j-b_j^*\varphi_j^*)=
\sum_{j=1}^n(\varphi_jb_j-b_j\varphi_j+b_j\varphi_j-\varphi_j^*b_j^*+\varphi_j^*b_j^*-b_j^*\varphi_j^*)$$
so that
$$2\psi-\sum_j(\varphi_jb_j-b_j^*\varphi_j^*)=$$
$$2\psi-\sum_j[\varphi_j,b_j]-\sum(b_j\varphi_j-\varphi_j^*b_j^*)+2\psi-2\psi-\sum_j[\varphi_j^*,b_j^*]$$
and
$$4\psi-\sum_j[\varphi_j,b_j]-\sum_j[\varphi_j^*,b_j^*]=$$
$$2\psi-\sum_j(\varphi_jb_j-b_j^*\varphi_j^*)+2\psi+\sum_j(b_j\varphi_j-\varphi_j^*b_j^*).$$
Thus, by (\ref{eq:511}) and (\ref{eq:512}),
$$
\left |4\psi([a,x])-\sum_j[\varphi_j,b_j]([a,x])-\sum_j[\varphi_j^*,b_j^*]([a,x])
\right |\le
$$
$$
\left | 2\psi([a,x])-\sum_j(\varphi_jb_j-b_j^*\varphi_j^*)([a,x])  \right |$$
$$+\left | 2\psi([a,x])+\sum_j(b_j\varphi_j-\varphi_j^*b_j^*)([a,x])  \right |<
10\epsilon\|a\|\|x\|.
$$
Since $[M,M]=M$, for any $z\in Z(M)$, if $z=\sum_k[a_k,x_k]$ we have 
$$
|4\psi(z)|\le 10\epsilon\sum_k\|x_k\|\|a_k\|,
$$ proving that $\psi(z)=0$.
\end{proof}

Variants of the argument used in the preceding proof will be used in Proposition~\ref{prop:0418131}(b) and in Proposition~\ref{prop:0326141}.
The proof of the following lemma is contained in \cite[Lemma 3.2]{HoPerRus12}.

\begin{lemma}\label{lem:0815131}
If $M$ is a von Neumann algebra, and $\psi\in M_*$ satisfies $\psi^*=-\psi$, then $D_\psi$ is a self-adjoint mapping.   Conversely, if $M$ is properly infinite and $D_\psi$ is self-adjoint, then $\psi^*=-\psi$.
\end{lemma}

\begin{theorem}\label{thm:0411131}
Let $M$ be a von Neumann algebra.
\begin{description}
\item[(a)] If every triple derivation of $M$ into $M_*$ is approximated in norm by inner triple derivations, then $M$ is finite.
\item[(b)] If $M$ is a finite von Neumann algebra acting on a separable Hilbert space or if $M$ is a finite factor, then every triple derivation of $M$ into $M_*$ is approximated in norm by inner triple derivations.
\end{description}
\end{theorem}
\begin{proof}

(a)  Assume that every triple derivation of $M$ into $M_*$ is a norm limit of inner such derivations and also assume for the moment that $M$ is properly infinite.  If $\psi\in M_*$ satisfies $\psi^*=-\psi$, then by Lemma~\ref{lem:0815131}, and (\ref{eq:0822131}),  $D_\psi\circ *\in \D_t(M,M_*)$. 
Then
 by Lemma~\ref{lem:0411132},   $\psi(1)=0$.   This is a contradiction if we take $\psi=i\phi_0$ where $\phi_0$ is any normal state of $M$. This proves that $M$ cannot be properly infinite.   

If $M$ is arbitrary, write $M=pM+(1-p)M$ for some central projection $p$, where $pM$ is finite and $(1-p)M$ is properly infinite.   It is easy to see that  if $\delta\in\D_t(M,M_*)$, then 
$p\delta\in \D_t(pM,(pM)_*)$ and similarly for $(1-p)\delta$ and that if $\Inn_t(M,M_*)$ is norm dense in $\D_t(M,M_*)$, then 
$\Inn_t(pM,(pM)_*)$ is norm dense in $\D_t(pM,(pM)_*)$, and $\Inn_t((1-p)M,((1-p)M)_*)$ is norm dense in $\D_t((1-p)M,((1-p)M)_*)$.  By the preceding paragraph, $1-p=0$, so that $M$ is finite.

(b)
Suppose first that $M$ is a finite factor.
Let $\psi\in M_*$ be such that  the inner derivation $D_\psi:x\mapsto \psi\cdot x-x\cdot\psi$, is self adjoint, that is, $D_\psi\in \Inn_b^*(M,M_*)$. By the proof of Lemma~\ref{lem:0815131} (namely, \cite[Lemma 3.2]{HoPerRus12}), $\psi^*=-\psi$ on $[M,M]$. Let us assume temporarily that $\psi(1)\in i\IR$, so that $\psi^*=-\psi$ on $M=\IC 1+[M,M]$. We also assume, temporarily, that $\psi=\hat x_\psi$  for some $x_\psi\in M$, that is, $\psi(y)=\tr(yx_\psi )$ for $y\in M$, where $\tr$ is a faithful normal finite trace on $M$.

We then have
\begin{equation}\label{eq:0323131}
x_\psi=\tr(x_\psi) 1+\sum_j[a_j+ib_j,c_j+id_j]
\end{equation}
where $a_j,b_j,c_j,d_j$ are self adjoint elements of $M$.
Expanding the right side of (\ref{eq:0323131}) and using the fact that $x_\psi^*=-x_\psi$, we have
\[
x_\psi=\tr(x_\psi) 1+\sum_j([a_j,c_j]-[b_j,d_j])
\]
so that
\[
\hat x_\psi=\tr(x_\psi) \tr(\cdot)+\sum_j([a_j,c_j]^{\widehat\  }-[b_j,d_j]^{\widehat\ }\, ).
\]
It is easy to check that for $a=a^*,b=b^*,x,y\in M$,
\[
[a,b]^{\widehat\ }([x^*,y])=\tpc{\hat a}{2b}{x}(y)-\tpc{2b}{\hat a}{x}(y).
\]
Thus
$$
D_\psi(x^*)(y)=\psi(x^*y-yx^*)=\tr\left(\sum_j\left([a_j,c_j]-[b_j,d_j]\right)[x^*,y]\right)
$$
so that \begin{equation}
D_\psi\circ *=\sum_j\left(L(\hat a_j,2c_j)-L(2c_j,\hat a_j)-L(\hat b_j,2d_j)+L(2d_j,\hat b_j)\right)\end{equation}
belongs to $ \Inn_t(M,M_*)$.

By replacing $\psi$ by $\psi^{\prime}=\psi-\Re\psi(1)\, \tr(\cdot)$, so that  $\psi'(1)\in i\IR$ and $D_\psi=D_{\psi^{\prime}}$, we now have that
if $\psi=\hat x_\psi$ for some $x_\psi\in M$, then $D_\psi\circ *
\in \Inn_t(M,M_*)$.  Since elements of the form $\hat x$ are dense in $M_*$ and $\norm{D_\psi}\le 2\norm{\psi}$, it follows that for every $\psi\in M_*$,  $D_\psi\circ *$ belongs to the norm closure of $\Inn_t(M,M_*)$.  From (\ref{eq:0822131}), $\Inn_t(M,M_*)$ is norm dense in 
$\mathcal{D}_t(A, A^*)$.  

Now suppose that $M$ is a finite von Neumann algebra acting on a separable Hilbert space, so that it admits a faithful normal finite trace $\tr$. Let $\psi\in M_*$ be such that  the inner derivation $D_\psi:x\mapsto \psi\cdot x-x\cdot\psi$, is self adjoint, that is, $D_\psi\in \Inn_b^*(M,M_*)$. As above, $\psi^*=-\psi$ on $[M,M]$.  Assuming temporarily that $\psi^*=-\psi$ on $Z(M)$, the above proof in the factor case shows that $\Inn_t(M,M_*)$ is norm dense in 
$\mathcal{D}_t(A, A^*)$.

To reduce to the case that $\psi^*=-\psi$ on $Z(M)$, we use direct integrals.
 Write $M$ as a direct integral of factors: $M=\int^\oplus_\Omega M(\omega)\, d\mu(\omega)$, where $M(\omega)$  is a finite factor with canonical trace $\tr_\omega$, which is an element of $M(\omega)_*$. 
 Since $M_*=\int^\oplus_\Omega M(\omega)_*\, d\mu(\omega)$ (\cite[p.285]{TakesakibookI}) we have $\psi=\int^\oplus_\Omega \psi(\omega)\, d\mu(\omega)$ for suitable $\psi(\omega)\in M(\omega)_*$.  It follows that 
if $x\in Z(M)$ then $\psi^*(x)=-\psi(x)$ if and only if $\tr_\omega(\psi(\omega))$ is purely imaginary for almost every $\omega$.  Here we are using the same notation for the extension of $\tr_\omega$ to a linear functional on $M(\omega)_*$   (\cite[p.174]{TakesakibookII}).
By replacing $\psi$ by $\psi^{\prime}$, where $\psi^{\prime}(\omega)=\psi(\omega)-\Re\tr_\omega(\psi(\omega))\tr_\omega$, so that $D_\psi=D_{\psi^{\prime}}$, we now have the desired result.
\end{proof}

Variants of the argument used in the preceding proof of (b) will be used in Theorem~\ref{thm:0416131} and Propositions~\ref{prop:0418131}(a) and  ~\ref{prop:0326141}.

\begin{corollary}
 If $M$ acts on a separable Hilbert space, or if $M$ is a factor, then $M$ is finite if and only if every triple derivation of $M$ into $M_*$ is approximated in norm by inner triple derivations. 
 \end{corollary}

\subsection{Zero-One law for factors}

\begin{theorem}\label{thm:0416131}  
If $M$ is a properly infinite factor, then the real vector space of triple derivations of $M$ into $M_*$, modulo the norm closure of the inner triple derivations,  has dimension 1.
\end{theorem}
\begin{proof}
Let $D_\psi\in \Inn_b^*(M,M_*)$ so that again by
the proof of Lemma~\ref{lem:0815131} (namely, \cite[Lemma 3.2]{HoPerRus12}), since $M=[M,M]$, we have $\psi^*=-\psi$ and so $\psi(1)=i\lambda$ for some $\lambda\in \IR$.  Write, by (\ref{eq:0822132}),
\begin{equation}\label{eq:0411131}
\psi=\varphi+i\lambda \phi_0
\end{equation} with $\varphi(1)=0$.  By Lemma~\ref{lem:0411131}, for every $\epsilon>0$, there exist $\varphi_j\in M_*$ and $b_j\in M$, such that with $\varphi_\epsilon=\sum_j[\varphi_j,b_j]$, we have 
$\|\varphi-\varphi_\epsilon\|<\epsilon$.  Since $\varphi^*=-\varphi$ we may assume $\varphi_\epsilon^*=-\varphi_\epsilon$.

If we write $\varphi_j=\xi_j+i\eta_j$ and $b_j=c_j+id_j$ where $\xi_,,\eta_j,c_j,d_j$ are selfadjoint, then it follows from $\varphi_\epsilon^*=-\varphi_\epsilon$ that $$\varphi_\epsilon=\sum_j([\xi_j,c_j]-[\eta_j,d_j]).$$
Further calculation shows that for all $x\in M$,
\[
D_{\varphi_\epsilon}(x^*)=\sum_j(\tpc{\xi_j}{2c_j}{x}-\tpc{2c_j}{\xi_j}{x}-\tpc{\eta_j}{2d_j}{x}+\tpc{2d_j}{\eta_j}{x}).
\]
This shows that $D_{\varphi_\epsilon}\circ *\in \Inn_t(M,M_*)$ so that $D_\varphi\circ *$ belongs to the norm closure of $\Inn_t(M,M_*)$. 

According to (\ref{eq:0822131}), every $\delta\in \D_t(M,M_*)$ has the form $\delta=\delta_0+\delta_1$, where $\delta_0=D_\psi\circ *$ is selfadjoint, and $\delta_1\in\Inn_t(M,M_*)$ is the inner triple derivation $\frac{1}{2}L(\delta(1),1)-\frac{1}{2}L(1,\delta(1))$. 
Lemma~\ref{lem:0411132} shows now that the map
\[
\delta+\overline{\Inn_t(M,M_*)}\mapsto \lambda
\]
is an isomorphism $$\D_t(M,M_*)/\overline{\Inn_t(M,M_*)}\sim\IR,$$
where $\lambda$ is defined by (\ref{eq:0411131}).

Explicitly, we define a map $\Phi:\D_t(M,M_*)/\overline{\Inn_t(M,M_*)}\rightarrow \IR$ as follows.  If $\delta\in\D_t(M,M_*)$, 
say $\delta=D_\psi\circ *+\delta_1$ as above, and $[\delta]=\delta+\overline{\Inn_t(M,M_*)}$, let $\Phi([\delta])=-i\psi(1)\in \IR$.  It follows from Lemmas~\ref{lem:0815131} and \ref{lem:0411132} that
$\Phi$ is well defined,  and it is easily seen to be linear, onto and one to one.  

Explicitly, if $\lambda\in\IR$ and we let $\psi=i\lambda\phi_0$ where $\phi_0$ is any normal state, then $\Phi([D_\psi\circ *])=\lambda$.   Also, if $\Phi([\delta])=0$ where $\delta=D_\psi\circ *+\delta_1$, then $\psi(1)=0$ and by the first part of the proof, 
$D_\psi\circ *\in \overline{\Inn(M,M_*)}$, so that $\delta\in \overline{\Inn(M,M_*)}$.
\end{proof}

\begin{corollary}
 If $M$ is a factor,
the linear space of 
 triple derivations into the predual, modulo the \underline {norm closure} of the inner triple derivations, has dimension 0 or 1: It is zero if the factor is finite; and it is 1 if the factor is infinite.
\end{corollary}

\section{Sums of commutators in the predual}


For any von Neumann algebra $M$, we shall write $(M_*)_0$ for the set of elements $\psi\in M_*$ such that $\psi(1)=0$ (in (\ref{eq:0822131}) we called this space $\{1\}_\perp$).  If $M$ is finite and admits a faithful normal finite trace $\tr$, which therefore extends to a trace on $M_*$ (\cite[p.174]{TakesakibookII}), then $(M_*)_0=\tr^{-1}(0)$.

\subsection{Factors of type $II_1$}

Recall that a finite factor $M$ of type I is both normally ternary weakly amenable and satisfies 
$(M_*)_0=[M_*,M]$.  For a finite factor of type II, the corresponding statements with $[M_*,M]$ and $\Inn_t(M,M_*)$ replaced
by their norm closures are also both true by Theorem~\ref{thm:0411131} and Lemma~\ref{lem:0411131}.  More is true in the type II case.

\begin{proposition}\label{prop:0418131}
Let $M$ be a finite von Neumann algebra. 
\begin{description}
\item[(a)] If $M$ acts on a separable Hilbert space or is a factor (hence admits a faithful normal finite trace), and if  $(M_*)_0=[M_*,M],
$ then $M$ is normally ternary weakly amenable.
\item[(b)] If $M$ is a factor and $M$ is normally ternary weakly amenable, then
$(M_*)_0=[M_*,M].
$
\end{description}
\end{proposition}
\begin{proof}
(a) Assume first that $M$ is a factor.
 Suppose  that $(M_*)_0=[M_*,M]$.
Let $\psi\in M_*$ be such that  the inner derivation $D_\psi:x\mapsto \psi\cdot x-x\cdot\psi$, is self adjoint, that is, $D_\psi\in \Inn_b^*(M,M_*)$. By the proof of Lemma~\ref{lem:0815131}, $\psi^*=-\psi$ on $[M,M]$. Let us assume temporarily that $\psi(1)\in i\IR$, so that $\psi^*=-\psi$ on $M=\IC 1+[M,M]$. 

By our assumption, we then have (here $\tr$ is the extended trace, \cite[p.174]{TakesakibookII})
\begin{equation}\label{eq:0401131bis}
\psi=\tr(\psi) \tr+\sum_j[\xi_j+i\eta_j,c_j+id_j]
\end{equation}
where $\xi_j,\eta_j$ are self adjoint elements of $M_*$ and $c_j,d_j$ are self adjoint elements of $M$.
Expanding the right side of (\ref{eq:0401131bis}) and using the fact that $\psi^*=-\psi$, we have
\[
\psi=\tr(\psi) \tr+\sum_j([\xi_j,c_j]-[\eta_j,d_j]).
\]
It is easy to check that for $\xi=\xi^*\in M_*$ and $c=c^*,x,y\in M$,
\[
[\xi,c]([x^*,y])=\tpc{ \xi}{2c}{x}(y)-\tpc{2c}{\xi}{x}(y).
\]
Thus
 \begin{equation}
 D_\psi\circ *=\sum_j\left(L(\xi_j,2c_j)-L(2c_j,\xi)-L(\eta_j,2d_j)+L(2d_j,\eta_j)\right),\end{equation}
 which belongs to $\Inn_t(M,M_*)$.

By replacing $\psi$ by $\psi^{\prime}=\psi-\Re\psi(1)\, \tr(\cdot)$, so that $D_\psi=D_{\psi^{\prime}}$, we now have that
for every $\psi$,  $D_\psi\circ *
\in \Inn_t(M,M_*)$.    From (\ref{eq:0822131}), $\mathcal{D}_t(A, A^*)=\Inn_t(M,M_*)$ proving that $M$ is normally ternary weakly amenable.

The case where $M$ is not necessarily a factor, but acts on a separable Hilbert space is proved  using direct integrals as in the proof of Theorem~\ref{thm:0411131}(b).

(b) Suppose that  $M$ is a finite factor and that $M$ is normally ternary weakly amenable.  Let $\psi\in M_*$ with $\tr(\psi)=\psi(1)=0$.  Suppose first that $\psi^*=-\psi$ so that $D_\psi$ is self adjoint and therefore $D_\psi\circ *$  belongs to $\D_t(M,M)$.   By our assumption, there exist $\varphi_j\in M_*$ and $b_j\in M$ such that
$D_\psi\circ *=\sum_{j=1}^n (L(\varphi_j,b_j)-L(b_j,\varphi_j))$ on $M$.\smallskip

For $x,a\in M$, direct calculations yield
$$ \psi(a^*x-xa^*)=\frac{1}{2}\sum_{j=1}^n(\varphi_jb_j-b_j^*\varphi_j^*)(a^*x)
+\frac{1}{2}\sum_{j=1}^n(b_j\varphi_j-\varphi_j^*b_j^*)(xa^*).$$

We  set $x=1$  to get
\begin{equation}\label{eq:6810}
0=\frac{1}{2}\sum_{j=1}^n(\varphi_jb_j-b_j^*\varphi_j^*)(a^*)
+\frac{1}{2}\sum_{j=1}^n(b_j\varphi_j-\varphi_j^*b_j^*)(a^*),
\end{equation}
and therefore
\begin{equation}
\psi(a^*x-xa^*)=\frac{1}{2}
\sum_{j=1}^n(\varphi_jb_j-b_j^*\varphi_j^*)(a^*x-xa^*),\end{equation} for every $a,x\in M$.\smallskip

Since $M=\IC 1+[M,M]$ and $\psi(1)=0$ it follows from (\ref{eq:6810}) (with $a=1$) that 
$$\psi=\frac{1}{2}\sum_{j=1}^n(\varphi_jb_j-b_j^*\varphi_j^*)=\frac{1}{2}\sum_{j=1}^n(\varphi_j^*b_j^*-b_j\varphi_j).$$
Hence
\begin{eqnarray*}
2\psi&=&\sum_{j=1}^n(\varphi_jb_j-b_j\varphi_j+b_j\varphi_j-\varphi_j^*b_j^*+\varphi_j^*b_j^*-b_j^*\varphi_j^*)\\
&=&\sum_{j=1}^n[\varphi_j,b_j]-2\psi+\sum_{j=1}^n[\varphi_j^*,b_j^*],
\end{eqnarray*}
which shows that $\psi\in [M_*,M]$.

Now let $\psi\in (M_*)_0$ and  write $\psi=\psi_1+\psi_2$, where $\psi_1^*=\psi_1$ and $\psi_2^*=-\psi_2$. Since $0=\tr(\psi)=\tr(\psi_1)+\tr(\psi_2)$ and $\tr(\psi_1)=\psi_1(1)$ is real and $\tr(\psi_2)=\psi_2(1)$ is purely imaginary,  $\tr(\psi_1)=0=\tr(\psi_2)$.   By the previous paragraph, $i\psi_1,\psi_2\in [M_*,M]$ and so $\psi=-i(i\psi_1)+\psi_2\in [M_*,M]$, completing the proof.
\end{proof}

\begin{corollary}
Let $M$ be a factor of type $II_1$.  Then
$
M$ is normally ternary weakly amenable if and only if $(M_*)_0=[M_*,M].
$
\end{corollary}

After proving Proposition~\ref{prop:0418131},  we learned from Ken Dykema that  \cite[Theorem 4.6]{DykKal05} gives a necessary and sufficient condition, in terms of its spectral decomposition as a normal operator affiliated with $M$, for an element in $M_*$ (where $M$ is  a $II_1$ factor) to belong to $[M_*,M]$, and that the same holds for a factor of type $II_\infty$ by using \cite[Theorem 4.7]{DykKal05}.  Using that criteria, it can be shown that $(M_*)_0\ne [M_*,M]$.
Hence we have

\begin{corollary}\label{cor:0821131}
A factor of type $II_1$ is never normally ternary weakly amenable.
\end{corollary}

No infinite factor can be approximately normally ternary weakly amenable by 
Theorem~\ref{thm:0411131}, much less normally ternary weakly amenable.  
As for  the case of a factor $M=B(H)$ of type $I_\infty$, we also have $(M_*)_0\ne [M_*,M]$, due to the work of Gary Weiss: \cite[Main Theorem]{Weiss86},\cite[Theorem 10]{Weiss80},\cite[Theorem 2.1]{Weiss04}(See Problem~4).  

 
\subsection{Finite von Neumann algebras of type I}

Now let $M$ be a finite von Neumann algebra of type $I_n$, with $n<\infty$. Then we can assume $$M=L^\infty(\Omega, \mu, M_n(\IC))=M_n(L^\infty(\Omega,\mu)),$$ $$ M_*=L^1(\Omega, \mu,M_n(\IC)_*))=M_n(L^1(\Omega,\mu))$$ and
$$Z(M)=L^\infty(\Omega,\mu)1.$$

It is known that the center valued trace is given by
$$\hbox{TR}\, (x)=\frac{1}{n}(\sum_1^n x_{ii})1\quad , \quad \hbox{ for }x=[x_{ij}]\in M
$$ (see \cite[p.1405]{Pearcy62} or \cite[p.65]{FacdelaHar80}).  We thus define, for a finite von Neumann algebra of type $I_n$ which has a faithful normal finite trace $\tr$,
$$\hbox{TR}\, (\psi)=\frac{1}{n}(\sum_1^n \psi_{ii})\tr\quad , \quad \hbox{ for }\psi=[\psi_{ij}]\in M_*.
$$

By modifying slightly the proof of Proposition~\ref{prop:0418131}  we have the following proposition. First we note the following elementary remark.

\begin{remark}
Let $M$ be a finite von Neumann algebra of type $I_n$ admitting a faithful normal finite trace and let $\psi\in M_*$.
\begin{description}
\item[(a)] If  $\TR(\psi)=0$, then $\psi$ vanishes on the center $Z(M)$ of $M$. 
\item[(b)] $\psi^*=-\psi$ on $Z(M)$ if and only if $\tr(\psi(\omega))$ is purely imaginary for almost every $\omega$.
\end{description}
\end{remark}
\begin{proof} If $x\in Z(M)$, then $x=f\cdot 1$ with $f\in L^\infty$ and 
\begin{eqnarray*}
\psi(x)&=& \int_\Omega\langle \psi(\omega),x(\omega)\rangle\, d\mu(\omega)=\int_\Omega\tr(\psi(\omega)x(\omega))\, d\mu(\omega)\\
&=&\int_\Omega\tr([\sum_k\psi_{ik}(\omega)x_{kj}(\omega)])\, d\mu(\omega)\\
&=&\int_\Omega(\sum_i\sum_k\psi_{ik}(\omega)x_{ki}(\omega))\, d\mu(\omega)\\
&=&\int_\Omega(\sum_k\psi_{kk}(\omega)x_{kk}(\omega))\, d\mu(\omega)\\
&=&\int_\Omega(\sum_k\psi_{kk}(\omega))f(\omega)\, d\mu(\omega)\\
&=&\int_\Omega(\TR \psi)(\omega) f(\omega)\, d\mu(\omega),
\end{eqnarray*}
proving (a).  As for (b), use $\psi(f\cdot 1)=\int_\Omega f(\omega)\tr(\psi(\omega))\, d\mu(\omega).$\end{proof}

\begin{proposition}\label{prop:0326141}
Let $M$ be a finite von Neumann algebra of type $I_n$ with $n<\infty$, which admits a faithful normal finite trace (equivalently, $M$ is countably decomposable, also called $\sigma$-finite).  Then $M$ is normally ternary weakly amenable if and only if $$\hbox{TR}^{-1}(0)=[M_*,M].$$  
\end{proposition}

\begin{proof} 

Suppose that $M$, a finite von Neumann algebra of type $I_n$ admitting a faithful normal finite trace, is normally ternary weakly amenable.   Let $\psi\in M_*$ with $\TR(\psi)=0$.  Suppose first that $\psi^*=-\psi$ so that $D_\psi$ is self adjoint and therefore $D_\psi\circ *$  belongs to $\D_t(M,M)$.   By our assumption, there exist $\varphi_j\in M_*$ and $b_j\in M$ such that
$D_\psi\circ *=\sum_{j=1}^n (L(\varphi_j,b_j)-L(b_j,\varphi_j))$ on $M$.\smallskip

For $x,a\in M$, direct calculations yield
$$ \psi(a^*x-xa^*)=\frac{1}{2}\sum_{j=1}^n(\varphi_jb_j-b_j^*\varphi_j^*)(a^*x)
+\frac{1}{2}\sum_{j=1}^n(b_j\varphi_j-\varphi_j^*b_j^*)(xa^*).$$

We  set $x=1$  to get
\begin{equation}\label{eq:6811}
0=\frac{1}{2}\sum_{j=1}^n(\varphi_jb_j-b_j^*\varphi_j^*)(a^*)
+\frac{1}{2}\sum_{j=1}^n(b_j\varphi_j-\varphi_j^*b_j^*)(a^*),
\end{equation}
and therefore
\begin{equation}
\psi(a^*x-xa^*)=\frac{1}{2}
\sum_{j=1}^n(\varphi_jb_j-b_j^*\varphi_j^*)(a^*x-xa^*),\end{equation} for every $a,x\in M$.\smallskip

Since $M=Z(M)+[M,M]$ and $\psi(Z(M))=0$ (since $\hbox{TR}\, \psi=0$) and $\sum_{j=1}^n(\varphi_jb_j-b_j^*\varphi_j^*)(Z(M))=0$ (by (\ref{eq:6811})) it follows that 
$$\psi=\frac{1}{2}\sum_{j=1}^n(\varphi_jb_j-b_j^*\varphi_j^*)=\frac{1}{2}\sum_{j=1}^n(\varphi_j^*b_j^*-b_j\varphi_j).$$
Hence
\begin{eqnarray*}
2\psi&=&\sum_{j=1}^n(\varphi_jb_j-b_j\varphi_j+b_j\varphi_j-\varphi_j^*b_j^*+\varphi_j^*b_j^*-b_j^*\varphi_j^*)\\
&=&\sum_{j=1}^n[\varphi_j,b_j]-2\psi+\sum_{j=1}^n[\varphi_j^*,b_j^*],
\end{eqnarray*}
which shows that $\psi\in [M_*,M]$.

Now let $\psi\in \TR^{-1}(0)$ and  write $\psi=\psi_1+\psi_2$, where $\psi_1^*=\psi_1$ and $\psi_2^*=-\psi_2$. Since $0=\TR(\psi)=\TR(\psi_1)+\TR(\psi_2)$ and $\TR(\psi_1)$ is real valued and $\TR(\psi_2)$ is purely imaginary valued,  $\TR(\psi_1)=0=\TR(\psi_2)$.   By the previous paragraph, $i\psi_1,\psi_2\in [M_*,M]$ and so $\psi=-i(i\psi_1)+\psi_2\in [M_*,M]$, completing the proof of one direction.

To prove the other direction, suppose that  $\TR^{-1}(0)=[M_*,M]$.
Let $\psi\in M_*$ be such that  the inner derivation $D_\psi:x\mapsto \psi\cdot x-x\cdot\psi$, is self adjoint, that is, $D_\psi\in \Inn_b^*(M,M_*)$. By Lemma~\ref{lem:0815131}, $\psi^*=-\psi$ on $[M,M]$. Let us assume temporarily that  $\psi^*=-\psi$ on $Z(M)$ so that $\psi^*=-\psi$. 


By our main assumption, we  have
\begin{equation}\label{eq:0401131}
\psi=\TR(\psi) +\sum_j[\varphi_j+i\xi_j,c_j+id_j]
\end{equation}
where $\varphi_j,\xi_j$ are self adjoint elements of $M_*$ and $c_j,d_j$ are self adjoint elements of $M$.
Expanding the right side of (\ref{eq:0401131}) and using the fact that $\psi^*=-\psi$, we have
\[
\psi=\TR(\psi) +\sum_j([\varphi_j,c_j]-[\xi_j,d_j]).
\]
It is easy to check that for $\varphi\in M_*$ and $c=c^*,x,y\in M$,
\[
[\varphi,c] ([x^*,y])=\tpc{\varphi}{2c}{x}(y)-\tpc{2c}{\varphi}{x}(y).
\]
Thus
 \begin{equation}
 D_\psi\circ *=\sum_j\left(L(\varphi_j,2c_j)-L(2c_j,\varphi_j)-L(\xi_j,2d_j)+L(2d_j,\xi_j)\right),\end{equation}
 which belongs to $\Inn_t(M,M_*)$.

By replacing $\psi$ by $\psi^{\prime}$, where $\psi^{\prime}(\omega)=\psi(\omega)-\Re\tr(\psi(\omega))1_n$, we have  $D_\psi=D_{\psi^{\prime}}$
and $(\psi')^*=-\psi'$ so that $D_\psi\circ *=D_{\psi'}\circ *
\in \Inn_t(M,M_*)$.    From (\ref{eq:0822131}), $\mathcal{D}_t(A, A^*)=\Inn_t(M,M_*)$ proving that $M$ is normally ternary weakly amenable.
\end{proof}

\begin{corollary} Let $M$ be a finite von Neumann algebra of type $I_n$ admitting a faithful normal finite trace $\tr$.
If $\tr^{-1}(0)=[M_*,M]$, then $M$ is normally ternary weakly amenable.
\end{corollary}

\begin{problem}
Does Theorem 1(b) hold for finite von Neumann algebras on non separable Hilbert spaces?
\end{problem}
\begin{problem}
Does
 Theorem 2 hold for all properly infinite von Neumann algebras?
\end{problem}

\begin{problem}\label{prob:3}
Is a finite countably decomposable  von Neumann algebra of type I normally ternary weakly amenable?
\end{problem}

\begin{problem}\label{prob:4}
For  a factor $M$ of type $III$, do we have $\{1\}_\perp=[M_*,M]$?
\end{problem}



\begin{bibdiv}
\begin{biblist}





\bib{BroPeaTopDuke68}{article}{
   author={Brown, Arlen}, 
   author={Pearcy, Carl},
   author={Topping, David},
      title={Commutators and the strong radical},
   journal={Duke Math. J.)},
   volume={35},
   date={1968},
   pages={853--859},
}

\bib{BunPas80}{article}{
   author={Bunce, J. W.}, 
   author={Paschke, W. L.},
         title={Derivations on a C$^*$-algebra and its
double dual},
   journal={J. Funct. Analysis},
   volume={37},
   date={1980},
   pages={235--247},
}

\bib{DykKal05}{article}{
   author={Dykema, Kenneth J.}, 
   author={Kalton, Nigel J.},
      title={Sums of commutators in ideals and modules of type II factors},
   journal={Ann. Inst. Fourier (Grenoble)},
   volume={55},
   number={3}
   date={2005},
   pages={931--971},
}





\bib{FacdelaHar80}{article}{
   author={Fack, Th.}, 
   author={de la Harpe, P.},
      title={Sommes de commutateurs dans les algbres de von Neumann finies continues},
   journal={Ann. Inst. Fourier (Grenoble)},
   volume={30},
   number={3},
   date={1980},
   pages={49--73},
}

\bib{Haagerup83}{article}{
   author={Haagerup, Uffe}, 
         title={All nuclear C$^*$-algebras are amenable},
   journal={Invent. Math.},
   volume={74},
   number={2},
   date={1983},
   pages={305--319},
}






\bib{HalmosAJM52}{article}{
   author={Halmos, Paul R.},
   title={Commutators of operators},
   journal={Amer. J. Math.},
   volume={74},
   date={1952},
   pages={237--240},
 }

\bib{HalmosAJM54}{article}{
   author={Halmos, Paul R.},
   title={Commutators of operators. II},
   journal={Amer. J. Math.},
   volume={76},
   date={1954},
   pages={191--198},
 }



\bib{HalpernTAMS69}{article}{
   author={Halpern, Herbert},
   title={Commutators in properly infinite von Neumann algebras},
   journal={Trans. Amer. Math. Soc.},
   volume={139},
   date={1969},
   pages={55--73},
 }



\bib{HoMarPerRus02}{article}{
   author={Ho, Tony}, 
   author={Martinez-Moreno, J.},
   author={Peralta, Antonio M.},
   author={Russo, Bernard},
   title={Derivations on real and complex JB$^\ast$-triples},
   journal={J. London Math. Soc.},
   volume={65},
   number={1},
   date={2002},
   pages={85--102},
}



\bib{HoPerRus12}{article}{
   author={Ho, Tony}, 
   author={Peralta, Antonio M.},
   author={Russo, Bernard},
   title={Ternary weakly amenable C$^*$-algebras and JB$^*$-triples},
   journal={Quarterly J. Math.},
  volume={64},
  note={First published online: November 28, 2012, doi: 10.1093/qmath/has032},
   number={4},
   date={2013},
   pages={1109--1139},
}


\bib{Kadison66}{article}{
   author={Kadison, Richard V. },
   title={Derivations of operator algebras},
   journal={Ann. of Math.},
   volume={83},
   number={2},
   date={1966},
   pages={280--293},
 }










\bib{Pearcy62}{article}{
   author={Pearcy, Carl}, 
      title={A complete set of unitary invariants for operators generating finite W*-algebras of type I},
   journal={Pacific J. Math.},
   volume={12},
   date={1962},
   pages={1405--1416},
}

\bib{PeaTop69}{article}{
   author={Pearcy, Carl}, 
   author={Topping, David},
   title={Commutators and certain $II_1$ factors},
   journal={J. Funct. Anal.},
   volume={3},
   date={1969},
   pages={69--78},
}


 
 
 
 \bib{PerRus10}{article}{
     author={Peralta, Antonio M.},
   author={Russo, Bernard},
   title={Automatic continuity of derivations on
C$^*$-algebras and JB$^*$-triples},
   journal={Journal of Algebra},
   volume={399},
   number={1},
   date={2014},
   pages={960--977},
}

 
 
\bib{Sakai66}{article}{
   author={Sakai, Shoichiro},
   title={Derivations of $W^{\ast} $-algebras},
   journal={Ann. of Math.},
   volume={83},
   number={2},
   date={1966},
   pages={273--279},
 }



\bib{TakesakibookI}{book}{
   author={Takesaki, M.},
   title={Theory of Operator Algebras I},
   series={Encyclopaedia of Mathematical Sciences},
   volume={124},
   note={Operator Algebras and Non-commutative Geometry, 6},
   publisher={Springer-Verlag},
   place={Berlin},
   date={2001},
   pages={xIx+415},
   isbn={978-3-540-42248-8},
}


\bib{TakesakibookII}{book}{
   author={Takesaki, M.},
   title={Theory of Operator Algebras II},
   series={Encyclopaedia of Mathematical Sciences},
   volume={125},
   note={Operator Algebras and Non-commutative Geometry, 6},
   publisher={Springer-Verlag},
   place={Berlin},
   date={2003},
   pages={xxii+518},
   isbn={3-540-42914-X},
   review={\MR{1943006 (2004g:46079)}},
}




\bib{Upmeier80}{article}{
   author={Upmeier, H.},
   title={Derivations of Jordan C$^*$-algebras},
   journal={Math. Scand.},
   volume={46},
   date={1980},
   pages={251--264},
 }

\bib{Weiss80}{article}{
   author={Weiss, Gary},
   title={Commutators of Hilbert Schmidt operators II},
   journal={Integral Equations and Operator Theory},
   volume={3/4},
   date={1980},
   pages={574--600},
}

\bib{Weiss86}{article}{
   author={Weiss, Gary},
   title={Commutators of Hilbert Schmidt operators I},
   journal={Integral Equations and Operator Theory},
   volume={9},
   date={1986},
   pages={877--892},
}





\bib{Weiss04}{article}{
   author={Weiss, Gary},
   title={$B(H)$-Commutators:  A Historical Survey},
   journal={Operator Theory: Advances and Applications},
   volume={153},
   date={2004},
   pages={307--329},
 }
 






\end{biblist}
\end{bibdiv}
\end{document}